\documentclass[12pt]{article}
\usepackage{latexsym,amsmath,amssymb}
\usepackage[dvipsnames,usenames,pdftex]{color}
\usepackage{graphicx,enumerate}

\newfont{\bb}{msbm10 at 12pt}

\def\r{\hbox{\bb R}}
\def\m{\hbox{\bb M}}
\def\h{\hbox{\bb H}}

\def\c{\hbox{\bb C}}
\def\s{\hbox{\bb S}}

\def\R{\hbox{\bb R}}

\def\grad2{\|\nabla h\|^4}

\def\hr{\mathbb{H}^2\times\mathbb{R}}
\def\sr{\mathbb{S}^2\times\mathbb{R}}

\newcommand{\hm}{\mathbb{E}(\kappa , \tau)}

\newcommand{\Qb}{\overline{Q}}
\newcommand{\campo}{\mathfrak{X}}
\newcommand{\set}[1]{\left\{#1\right\}}
\newcommand{\To}{\longrightarrow}
\newcommand{\norm}[1]{\left\Vert #1 \right\Vert}
\newcommand{\eps}{\varepsilon}
\newcommand{\beq}{\begin{equation}}

\newcommand{\eeq}{\end{equation}}

\newcommand{\zb} {\bar{z}}
\newcommand{\abs}[1]{\left\vert #1 \right\vert}
\newcommand{\mm}{\hbox{\bb{M}}^2(\varepsilon)\times\hbox{\bb{R}}}

\def\pz{\frac{\partial\ }{\partial z}}
\def\pzb{\frac{\partial\ }{\partial \overline{z}}}

\usepackage[latin1]{inputenc}
\usepackage {amsmath}
\usepackage {amsthm}
\usepackage{times}
\usepackage{amscd}
\usepackage{epsf}

\begin{document}

\theoremstyle{plain}\newtheorem{lema}{Lemma}
\theoremstyle{plain}\newtheorem{definicion}{Definition}
\theoremstyle{plain}\newtheorem{proposicion}{Proposition}
\theoremstyle{plain}\newtheorem{teorema}{Theorem}
\theoremstyle{plain}\newtheorem{ejemplo}{Example}
\theoremstyle{plain}\newtheorem{remark}{Remark}
\theoremstyle{plain}\newtheorem{corolario}{Corollary}

\hyphenation{di-ffe-ren-tia-ble}

\begin{center}
\rule{15cm}{1.5pt} \vspace{0.5cm}

{\Large \bf Totally umbilical disks and applications to surfaces \\[4mm]
in three-dimensional homogeneous spaces}\\

\vspace{.3cm}

{\large Jos\'{e} M. Espinar$\,^\dag$\footnote{The author is
partially supported by Spanish MEC-FEDER Grant MTM2007-65249, and
Regional J. Andalucia Grants P06-FQM-01642 and FQM325}, Isabel
Fern\'{a}ndez$\,^\ddag$ \footnote{The author is partially
supported by Spanish MEC-FEDER Grant MTM2007-64504, and Regional
J. Andalucia
Grants P06-FQM-01642 and FQM325}}\\
\vspace{0.3cm} \rule{15cm}{1.5pt}
\end{center}

\vspace{.5cm}

\noindent $\mbox{}^\dag$ Institut de Math\'{e}matiques, Universit\'{e} Paris VII, 2 place
Jussieu, 75005 Paris, France;\\ e-mail: jespinar@ugr.es

\noindent $\mbox{}^\ddag$ Departamento de Matem\'{a}tica Aplicada I, Universidad de Sevilla, 41012 Sevilla,
Spain;\\ e-mail: isafer@us.es


\vspace{.5cm}

\begin{center}
{\it This paper is in honor of Professor Manfredo do Carmo for its 80'th birthday.}
\end{center}

\vspace{.3cm}

\begin{abstract}
Following \cite{choe} and \cite{dCF}, we give sufficient conditions for a disk type
surface, with piecewise smooth boundary, to be totally umbilical for a given Coddazi
pair. As a consequence, we obtain rigidity results for surfaces in space forms and
in homogeneous product spaces that generalizes some known results.
\end{abstract}

\section{Introduction}

It is well known that a totally umbilical surface in $\r ^3$ is part of either  a round
sphere or a plane. Using this result, H. Hopf \cite{hopf} proved that an immersed
constant mean curvature (CMC) sphere in $\R^3$ must be a round sphere by
introducing a quadratic differential that turns out to be holomorphic on CMC
surfaces, and that vanishes at the umbilic points of the surface. Thus, the proof
follows from the fact that any holomorphic quadratic differential on a sphere must
vanish identically and the previous classification.

In the case of constant Gaussian curvature (CGC) surfaces, the Liebmann Theorem
states that the only complete surfaces with positive constant Gaussian curvature in
$\r ^3$ are the totally umbilical round spheres. T. K. Milnor \cite{Mi} gave a proof
of the Liebmann Theorem similar to the one due to Hopf. In fact, she proved that the
$(2,0)$-part of the first fundamental form of the surface is holomorphic for the
structure given by the second fundamental form if and only if the Gaussian curvature
is constant. Also, the zeroes of this quadratic differential are the umbilic points.

Thus, the existence of a certain holomorphic quadratic differential is the main tool
for the classification of immersed spheres with constant mean or Gaussian curvature.
The underlying idea for the construction of these two differentials relies on an
abstract structure, the {\em Codazzi pairs}, that can be defined on a differentiable
surface (for example, in the above situations, the Codazzi pair consists of the
first and second fundamental forms of the surface). Under some geometrical
conditions, a Codazzi pair gives rise to a holomorphic quadratic differential on the
surface that can be used to classify those surfaces that are topological spheres.

In this line, U. Abresch and H. Rosenberg (see \cite{AR} and \cite{AR2}) recently
showed the existence of such a differential for CMC surfaces in the homogeneous
spaces with a 4-dimensional isometry group. These homogeneous space are denoted by
$\hm$, where $\kappa$ and $\tau$ are constant and $\kappa -4\tau ^2 \neq 0$. They
can be classified as: the product spaces $\hr$ if $\kappa = -1$ and $\tau =0$, or
$\sr$ if $\kappa = 1$ and $\tau =0$, the Heisenberg space ${\rm Nil}_3$ if $\kappa =
0$ and $\tau = 1/2$, the Berger spheres $\s ^3 _{Berger}$ if $\kappa =1 $ and $\tau
\neq 0$, and the universal covering of ${\rm PSL}(2,\r)$ if $\kappa = -1$ and $\tau
\neq 0$. Using this differential they were able to classify all the immersed CMC
spheres in these spaces, putting the study of surfaces in homogeneous spaces in a
new light (see \cite{AR}, \cite{AR2}, \cite{CoR}, \cite{FM1}, \cite{FM2},
\cite{EGR}, \cite{ER} and references therein).

In the same spirit, J. A. Aledo, J. M. Espinar and J. A. G\'{a}lvez \cite{AEG1}
proved that for a large class of surfaces of CGC in $\hr$ and $\sr$ there exists a
Codazzi pair related with the the first fundamental form, second fundamental form
and height function. In addition, this pair has constant extrinsic curvature, which
gives the existence of a holomorphic quadratic differential for any surface of positive
constant curvature in $\h^2 \times \r$ and constant curvature greater than one in
$\s^2 \times \r$. Moreover, the holomorphic quadratic differential vanishes at the
\emph{umbilic points} of the introduced Codazzi pair, which allows them to classify
the complete CGC surfaces.\\

Regarding surfaces with non-empty boundary, a natural problem is to determine
whether such a surface is part of a totally umbilical sphere. In this way, J. Nitsche
\cite{N} showed that an immersed disk type CMC surface in $\r^3$ whose boundary is a
line of curvature must be a part of a totally umbilical surface. On the other hand,
J. A. G\'{a}lvez and A. Martínez \cite{GM} proved a Liebmann-type theorem for
immersed CGC disks in $\r ^3$ when the boundary is a line of curvature.

When the boundary is non regular, but piecewise differentiable, J. Choe \cite{choe} extended Nitsche's Theorem
under some additional conditions on the singular points at the boundary.
Its proof is based on the control of the zeroes of the holomorphic quadratic differential introduced by Hopf.

A similar result for surfaces in the product spaces $\s^2 \times \r$ and $\h^2
\times \r$ was proved by M.P. do Carmo and I. Fern\'{a}ndez in \cite{dCF}.
Nevertheless, in this work the mean curvature of the surface is not assumed to be constant, and therefore the Abresch-Rosenberg differential is not holomorphic. On the other hand, the required regularity of the surface is more restrictive than in \cite{choe}. The control of the zeroes Abresch-Rosenberg differential is achieved here by a more general condition on the mean curvature, and the use of the following result (see either \cite[Lemma 2.6.1, pp 70]{J} or \cite{HW} where the original proof is done, the version we use here is \cite[Lemma
2.7.1, pp 75]{J}):

\vspace{.2cm}

\begin{lema}\label{lem:cauchy}
Let $f:U\to\c$ be a differentiable function defined on a complex domain
$U\subset\c$, and suppose that there exists a continuous real-valued (necessarily
non negative) function $\mu$ such that
\begin{equation}\label{eq:cauchy}
|f_{\bar z}|\leq \mu(z)|f(z)|,\qquad \forall\,z\in U.
\end{equation}
Then either $f\equiv 0$ in $U$ or it has isolated zeroes. Moreover, if $z_0$ is an
isolated zero of $f$, there exists a positive integer $k$ such that locally around
$z_0$
\begin{equation}\label{eq:ceros}
f(z)=(z-z_0)^k g(z),
\end{equation}
where $g$ is a continuous non-vanishing function.
\end{lema}

Our aim in this work is to show that the results in can be adapted to the general
setting of Codazzi pairs (see Theorem \ref{th:choe} for a precise statement, and Theorems \ref{th:special} and \ref{th:KIproducto} for its apliccations).
In particular, applied to CMC surfaces in $\s^2\times \r$ or $\h ^2 \times \r$ this result gives the do
Carmo-Fern\'{a}ndez's Theorem. Moreover, it also generalizes the Martínez-G\'{a}lvez's Theorem to piecewise smooth boundary,
and extends to other homogeneous product spaces.\\

It should be remarked that Lemma \ref{lem:cauchy} has been also a fundamental tool in the
classification of topological spheres with either constant positive extrinsic
curvature in $\hr$ and $\sr$ (see \cite{EGR}), or constant mean curvature
$H>1/\sqrt{3}$ in ${\rm Sol}_3$ (see \cite{DM}). Furthermore, it would not be
surprising if Lemma \ref{lem:cauchy} is a key for the classification of topological
spheres with either constant positive extrinsic curvature or constant positive
Gaussian curvature in the other homogeneous spaces, or with constant mean curvature
$H>0$ in ${\rm Sol}_3$.

%
%
%

\section{Preliminaries}\label{sec:prelim}

First, we need to state some basic facts on Codazzi pairs (see \cite{Mi} for
details).

Let $\Sigma $ be an orientable surface and $(A,B)$ a pair of real quadratic forms on
$\Sigma$ such that $A$ is a Riemannian metric and $B$ a quadratic form, such a pair
$(A,B)$ is called a {\it fundamental pair}. Associated to a fundamental pair
$(A,B)$, we define the {\em shape operator} of the pair as the unique self-adjoint
endomorphism of vector fields $S$ on $\Sigma$ given by
$$ B(X,Y) = A(SX,Y), \quad X,Y \in\campo (\Sigma).$$

We define the mean curvature, the extrinsic curvature and the principal curvatures
of $(A,B)$ as half the trace, the determinant and the eigenvalues of the endomorphism
S, respectively.

In particular, given local parameters $(u,v)$ on $\Sigma$ such that
$$
A=E\,du^2+2F\,dudv+G\,dv^2,\qquad B=e\,du^2+2f\,dudv+g\,dv^2,
$$
the mean curvature and the extrinsic curvature of the pair are given, respectively,
by
\begin{eqnarray}
H(A,B)&=&\frac{E g-2F f+G e}{2(EG-F^2)}, \label{CM}\\
K_e(A,B)&=&\frac{eg-f^2}{EG-F^2},\label{CE}
\end{eqnarray}and the skew curvature
\begin{equation}\label{CS}
q(A,B)= H(A,B)^2 -K_e(A,B) .
\end{equation}

Associated to any fundamental pair $(A,B)$  we define the \emph{lines of curvature}
form $W=W(A,B)$ by
\begin{equation}
\sqrt{EG-F^2} W = (E f - F e)\, du^2 +(E g - G e)\,du dv +(F g-G f) dv^2 .
\end{equation}

The integral curves for the equation $W=0$ are level lines for doubly orthogonal
coordinates and they are called \emph{lines of curvature associated to the pair
$(A,B)$}. Umbilic points are those where $W \equiv 0$ for all values of $du$ and
$dv$ or, equivalently, if $q=0$ at this point. Thus, a fundamental pair $(A,B)$ is
said to be {\it totally umbilical} if $q(A,B)$ vanishes identically on $\Sigma$. If
we take $z = u+ i v$ a local conformal parameter for the Riemannian metric $A$, we
can define
$$Q=Q(A,B):=\frac{1}{4}(e-g -2 i f)\,dz ^2 ,$$
that is, $Q$ is the $(2,0)-$part of the complexification of $B$ for the conformal
structure given by $A$. Also, in this setting, it is easy to check that
$$ -2 \, {\rm Im} (Q) = W $$ thus, the \emph{lines of curvature} associated
with a fundamental pair $(A,B)$ are given by
$$ {\rm Im} (Q)=0 .$$


Now, we recall the definition of the {\it Codazzi Tensor and Function} associated to
any fundamental pair $(A,B)$ introduced in \cite{AEG3} (see also \cite{Es}).

\begin{definicion}\label{d1}
Given a fundamental pair $(A,B)$, we define its Codazzi Tensor as
$$ T_S : \campo (\Sigma) \times \campo (\Sigma) \To \campo (\Sigma),$$
\begin{equation*}
T_S (X,Y) = \nabla _X  SY - \nabla _Y SX - S [X,Y ] \, , \; X, Y \in \campo(\Sigma),
\end{equation*}here $\nabla $ stands for the Levi-Civita connection associated to
$A$ and $S$ denotes the shape operator.

Moreover, we define the {\bf Codazzi function associated to $S$} as
\begin{equation*}
\begin{matrix}
\mathcal{T}_S & : & \Sigma & \To & \r \\
\mbox{} & \mbox{} &  p & \longmapsto & \dfrac{\norm{T_S(X_p , Y_p)}^2}{\norm{X_p
\wedge Y_p}^2}
\end{matrix}
\end{equation*}where $X, Y \in \campo (U)$ are linearly independent vector fields
defined in a neighborhood $U$ of $p$ and
\begin{equation*}
\norm{X \wedge Y}^2  := \norm{X}^2\norm{Y}^2 - A({X},{Y})^2 ,
\end{equation*}$\norm{ \cdot }$ is the norm with respect to $A$.
\end{definicion}

It is easy to check that the Codazzi Tensor is skew-symmetric and the Codazzi
function is well defined.

\begin{definicion}
A fundamental pair $(A,B)$ is said to be a Codazzi pair if its associated Codazzi
Tensor vanishes identically. This means it satisfies  the classical Codazzi
equations for surfaces in a 3-dimensional space form, that is,
\begin{equation}
\label{ecuacionCodazzi}
e_v-f_u=e\Gamma_{12}^1+f(\Gamma_{12}^2-\Gamma_{11}^1)-g\Gamma_{11}^2,\quad
f_v-g_u=e\Gamma_{22}^1+f(\Gamma_{22}^2-\Gamma_{12}^1)-g\Gamma_{12}^2,
\end{equation}
where $\Gamma_{ij}^k$ are the Christoffel symbols for the Riemannian metric $A$
w.r.t. the Riemannian connection of $A$.
\end{definicion}

Now, we establish a result (proved in \cite{AEG3}, see also \cite{Es}) that relates
the Codazzi function associated to $(A,B)$ to the quadratic differential $Q$ defined
before.

\begin{lema}\label{l1.2}
Let $(A,B)$ be a fundamental pair on $\Sigma$ with associated shape operator $S$.
Set $\tilde{S} = S - H\, {\rm Id}$, where $H = H(A,B)$. Then
\begin{equation*}
\mathcal{T}_{\tilde{S}} =  2 \frac{|Q_{\zb}|^2}{\lambda ^3} ,
\end{equation*}or equivalently
\begin{equation*}
\abs{ Q _{\zb}}^2 = \frac{\lambda }{2} \frac{\mathcal{T}_{\tilde{S}}}{q} |Q|^2 ,
\end{equation*}where $z$ is a local conformal parameter for $A$, i.e., $A = 2 \lambda \,
|dz|^2$, and $Q$ is the $(2,0)-$part of the complexification
of B for the conformal structure given by $A$. Moreover, if $(A,B)$ is Codazzi, then
$$ \mathcal{T}_{\tilde{S}} = \norm{d H}^2 .$$
\end{lema}



\section{Totally umbilical disks}\label{sec:choe}

Our main theorem gives a sufficient condition for a disk-type surface (with non
regular boundary) to be totally umbilical for a given Codazzi pair. It is inspired
in the results of \cite{choe} and \cite{dCF}. We will give the proof of the theorem in Section \ref{sec:proof}.

\begin{teorema}\label{th:choe}
Let $\Sigma $ be a compact disk with piecewise differentiable boundary. We will call the \emph{vertices} of the surface to the finite set of non-regular boundary
points. Assume also that $\Sigma $ is contained as an interior set in a
differentiable surface $\tilde{\Sigma}$ without boundary.

Let $(A,B)$ be a Codazzi pair on $\tilde \Sigma$. Assume that the following
conditions are satisfied:
\begin{enumerate}
\item On $\tilde{\Sigma}$ we have $\,\norm{dH} \leq h \sqrt{q}\,$, where $H$ and $q$ are the
mean and skew curvature of the pair $(A,B)$, $h$ is a continuous function, and
$\norm{ \cdot }$ means the norm with respect to the metric $A$.
\item The number of vertices in $\partial \Sigma $ with angle $<\pi$ (measured with respect to
the metric $A$) is less than or equal to $3$.
\item The regular curves in $\partial \Sigma $ are lines of curvature for the pair $(A,B)$.
\end{enumerate}
Then, $\Sigma$ is totally umbilical for the pair $(A,B)$.
\end{teorema}

It should be remarked that, even the outline of the proof follows the ideas in \cite{choe}, our hypothesis on the regularity on the surfaces is a bit more restrictive. Indeed, in \cite{choe} the Codazzi pair is assumed to be extended around the points regular points of the boundary, but (possibly) not around the vertices. However, in that case the holomorphicity of the Hopf differential makes possible to control the behavior at these points, whereas in the general case of a non holomorphic differential $Q$ it is necessary to impose more regularity to obtain a estimate of the rotation index at the vertices. This is the same strategy employed in \cite{dCF}. 

As a particular case of Theorem \ref{th:choe} we can consider the
case of an immersed surface $\Sigma \subset \m ^2 (\eps )\times \r$, where
$\m^2(\eps)$ denotes the complete simply connected surface of constant curvature
$\eps\in \set{+1,0,-1}$.

Denote by $h$ the height function of the surface, $H$ its mean curvature,
and $(I,II)$ the first and second fundamental forms respectively. Then if we define
$$ B= 2H \, II - \eps dh ^2 + \frac{\eps}{2}\norm{\nabla h}^2 I ,$$
it is easy to check that the pair $(I,B)$ is a Codazzi pair when $H$ is
constant. Note that when $\eps = 0$, $B$ is nothing
but a constant multiple of the second fundamental form. Our theorem then implies the
result in \cite{dCF} when the surface has constant mean curvature (notice that in
this case the first hypothesis trivially holds).

\begin{remark}
We should point out that the $(2,0)-$part of the complexification of $B$ for the
conformal structure given by $I$ agrees with the Abresch-Rosenberg differential (see
\cite{AR} and \cite{AR2}). In these papers, it is proved that this
differential is holomorphic for the conformal structure given by $I$ if the mean
curvature, $H$, of the surface is constant.

Moreover, the gradient term in the definition of $B$ makes that $H(I,B)=2H^2$. Thus,
we can apply \cite[Lemma 6]{Mi} for ensuring:
\begin{center}
$(I,B)$ is Codazzi if, and only if, the Abresch-Rosenberg differential is
holomorphic.
\end{center} This is reason for adding the gradient term in the
definition of $B$, to use the Codazzi pair theory.
\end{remark}

To prove the result in \cite{dCF} in all generality, we establish Theorem
\ref{th:choe} in a more general version. Actually, we have imposed the condition to
be Codazzi in Theorem \ref{th:choe} because we can not control the Codazzi Tensor of
a surface immersed in a general three-manifold. But, it is still possible in a
homogeneous three-manifold since it depends on the mean curvature and the height
function.


\section{Applications to surfaces in homogeneous 3-manifolds}

Throughout this section, $\Sigma $ will denote a compact disk immersed in a
homogeneous three-manifold.


We will assume that $\Sigma$ has piecewise differentiable boundary. We will call the \emph{ vertices} of the
surface the (finite) set of non-regular boundary points. We will also assume that
$\Sigma $ is contained in the interior of a differentiable surface $\tilde{\Sigma}$
without boundary.

\subsection{Surfaces in space forms}

Let $\m^3(\eps)$ denote the complete simply connected Riemannian $3$-manifold of
constant curvature $\eps$. That is, $\m^3(\eps)$ is $\s ^3$ if $\eps =1$, $\r ^3$ if
$\eps =0$, or $\h ^3$ if $\eps =-1$. It is well known that for an immersed surface
$\Sigma$ in $\m^3(\eps)$ the first and second fundamental forms, $I$ and $II$, are
a Codazzi pair. Moreover, the mean, extrinsic and skew curvature, as well as the
lines of curvature for the pair $(I,II)$ agree with the usual definitions for an
immersed surface.

The classification of totally umbilical surfaces in these spaces is well-known (see
\cite{spi}). Roughly speaking, they are part of round spheres, planes, or
horospheres (in the case of $\m^3(-1)=\h^3$).

We say that an immersed surface $\Sigma\subset\m^3(\eps)$ is a {\em special
Weingarten} surface if its mean curvature is of the form $H= f(H^2-K_e)$ for a
smooth function $f$, where $K_e$ is the extrinsic curvature. This class of surfaces
includes the case of constant mean curvature surfaces for $H=f(q):=c$, as well as
the surfaces with constant positive extrinsic curvature (equivalently, constant
Gaussian curvature, since in this setting the Gaussian curvature, $K$, is
$K=K_e+\eps$) for $H=f(q):=\sqrt{q+c}$, in both cases $c$ is the positive constant so
that $H=c$ or $K_e=c$.

Then, we have:
\begin{teorema}\label{th:special}
Let $\tilde{\Sigma} \subset \m ^3 (\eps )$ be a special Weingarten surface, and $\Sigma\subset\tilde{\Sigma}$ a compact disk with the regularity
conditions assumed at the beginning of the Section. Assume also that the following conditions are satisfied:
\begin{enumerate}
\item The number of vertices in $\partial \Sigma $ with angle $<\pi$ is less than or equal to $3$.
\item The regular curves in $\partial \Sigma $ are lines of curvature.
\end{enumerate}
Then, $\Sigma $ is a piece of one of the totally umbilical surfaces described above.
\end{teorema}

\begin{proof}
Let us show that there exists a continuous function $h$ such that $\norm{dH} \leq h
\sqrt{q}$. Since $H=f(H^2-K_e)=f(q)$, then
$$dH=f'(q)\,dq=\frac{f'(q)}{\lambda^2}(d(|Q|^2)-d(\lambda^2)q),$$
where we have used that $|Q|^2=\lambda^2 q$. Now just observe that
$\norm{d(|Q|^2)}=\norm{d(Q\bar{Q})} \leq |Q| (\norm{dQ} + \norm{d\bar{Q}}) $, and so
the desired condition holds for $h=\frac{|f'(q)|}{\lambda}\big(\norm{dQ} +
\norm{d\bar{Q}} + 2\sqrt{q}\,d\lambda\big)$.

Thus, the result follows from Theorem \ref{th:choe} applied to the Codazzi pair
$(I,II)$ and the classification of totally umbilical surfaces in space forms
described in \cite{spi}.
\end{proof}

The above theorem generalizes the result in
\cite{GM} regarding positive constant Gaussian
curvature surfaces in $\R^3$ with everywhere regular boundary.

\subsection{$K$-surfaces in product spaces}

Now we will deal with immersed disks in the product space $\m^2(\varepsilon)\times
\r$, where $\m^2(\varepsilon)$ denotes $\s^2$ if $\varepsilon=1$, $\r ^2$ if $\eps
=0$ or $\h^2$ if $\varepsilon=-1$.

Let $\Sigma \subset \m^2(\varepsilon)\times\r$ be an orientable immersed surface
with unit normal vector field $N$ and let $I$ and $II$ be the first and second
fundamental forms, $K$ its Gaussian curvature. We denote by $h:\Sigma \to\R$ the
height function of the immersion, that is, the restriction to the surface of the
canonical projection $\mm\to\R$. It is easy to show that, if we denote by $\xi$ the
unit vertical vector field in $\mm$ (that is, the Killing field $\xi = \partial
_t$), then
\begin{equation}\label{eq:h}
\xi=\nabla h + \nu N,
\end{equation}where $\nu=\langle \xi,N\rangle$.\\

If we assume that $K\neq \varepsilon$ at every point on $\Sigma $, we can define the
new quadratic form
\begin{equation}\label{A}
A=I+\frac{\eps}{K-\eps}\,dh^2.
\end{equation}

Moreover, if $K>\mbox{max}\{0,\eps\}$ then $A$ is a Riemannian metric.

\begin{definicion}
Let $\Sigma\subset \m^2(\varepsilon)\times\r$ be an immersed surface with constant
Gaussian curvature $K$. We will say that $\Sigma$ is a $K-$surface in
$\m^2(\varepsilon)\times\r$ if $K>\mbox{max}\{0,\eps\}$.
\end{definicion}

In \cite{AEG1}, the authors show that the pair $(A,II)$ is Codazzi for a $K-$surface
and its extrinsic curvature is
$$K_e(A,II)=K-\eps ,$$
when $\eps \in \{-1,1\}$. Note that when $\eps =0$, then $\m ^2 (0)\times \r \equiv
\r^3$ and this case was discussed above. In this case, the $(2,0)$-part of $A$ with
respect to the conformal structure induced by $II$, $Q(II,A)$, is holomorphic on
$\Sigma$ (\cite[Lemma 8]{Mi}). This is the main tool for the classification given in
\cite{AEG1} of the complete $K$-surfaces. It turns out that these surfaces are
rotationally invariant spheres, extending the previously known classification in the
euclidean case.

\begin{teorema}\label{th:KIproducto}
Let $\, \tilde{\Sigma} \subset \m ^2 (\eps )\times \r\,$ be a $K-$surface, and $\,\Sigma\subset\tilde{\Sigma}\,$ a compact disk with the
regularity conditions stated at the beginning of the section. Assume that the
following conditions are satisfied:
\begin{enumerate}
\item The number of vertices in $\partial \Sigma $ with angle $<\pi$ (with respect to the metric $A$)
is less than or equal to $3$.
\item The regular curves in $\partial \Sigma $ are lines of curvature for the pair $(A,II)$.
\end{enumerate}
Then, $\Sigma $ is a piece of one of the rotational spheres described in \cite{AEG1}.
\end{teorema}

\begin{proof}
Since $\Sigma$ is a $K-$surface, $(A,II)$ is a Codazzi pair and
$c:=K_e(A,II)=K-\eps$ is a positive constant. Set $q:=q(A,II)$ and $H:=H(A,II)$,
then $H = f(q):=\sqrt{q+c}$. Thus, $\norm{d H} \leq h \sqrt{q}$, where $h$ is a
continuous function. Here, the norm $\norm{ \cdot }$ is with respect to $A$.

Thus, from Theorem \ref{th:choe}, $\Sigma$ is totally umbilical for the pair
$(A,II)$, and so also $Q(II,A)$ vanishes identically on $\Sigma$, giving that $\Sigma$ is
a piece of one of the complete examples described in \cite{AEG1}.
\end{proof}

In a recent paper \cite{ALP}, the authors deal with the regular case for
$K-$surfaces in product spaces, more precisely, they proved the following

\vspace{.3cm}

{\bf Theorem.} {\it Let $\Sigma \subset \m ^2 (\eps )\times \r$ be a compact surface
with positive constant Gaussian curvature $K>\eps $ and such that the imaginary part
of $Q(A,II)$ vanishes identically along $\partial \Sigma$, where $\partial \Sigma $
is a closed regular curve in $\m ^2 (\eps ) \times \r$. Then $\Sigma$ is a piece of
a rotational complete $K-$surface.}

\vspace{.3cm}

Thus, the above result is a particular case of Theorem \ref{th:KIproducto}.

Despite the case of surfaces in space forms, in the spaces $\mm$ the lines of curvature
for the Coddazi pair $(A,II)$  we work with do not agree in general with the
classical lines of curvature (except of course when $\eps=0$). In the following lemma
we give sufficient conditions on a curve in $\Sigma$ to be a line of curvature of
$(A,II)$. These conditions are inspired by those given in \cite{dCF} for the pair
$(I,B)$ defined in Section \ref{sec:choe}.

\begin{lema}\label{lem:ldc}
Let $\Sigma\subset\mm$ be a $K-$surface, and $\gamma\subset\Sigma$ a differentiable
curve. Assume that $\gamma$ satisfies one of the following conditions,
\begin{enumerate}
\item $\gamma$ is contained in a horizontal slice.
\item $\gamma$ is an integral curve of $\,\nabla h$, where $h$ is the
height function of $\Sigma$ (here $\nabla$ is the gradient with respect to $I$).
\end{enumerate}
Then $\gamma$ is a line of curvature for the pair $(A,II)$ if and only if it is a
line of curvature of $\Sigma$ in the classical sense.
\end{lema}

\begin{proof}
This follows from the fact that the orthogonal vectors to $\gamma'$ for the metrics
$A$ and $I$ agree. Indeed, if $\gamma$ is horizontal then $dh(\gamma')=0$, whereas
in the second case $dh(n)=0$ for any vector $n$ orthogonal to $\gamma'$ with respect
to $I$. Thus, by the very definition of $A$ (see Equation \eqref{A}) we are done.
\end{proof}

The following result is then a straightforward consequence of Theorem \ref{th:choe}
and the above lemma.

\begin{corolario}
Let $\Sigma \subset \m ^2 (\eps )\times \r$ be disk-type $K-$surface satisfying the
regularity conditions stated at the beginning of this section. Assume that the
following conditions are satisfied:
\begin{enumerate}
\item The number of vertices in $\partial \Sigma $ with angle $<\pi$ (with respect to the
induced metric $A$) is less than or equal to $3$.
\item Every regular component $\gamma$ of $\partial \Sigma $ is a line of curvature (in the
classical sense) satisfying one of the following properties:
 \begin{itemize}
\item $\gamma$ is contained in a horizontal slice.
\item $\gamma$ is an integral curve of $\,\nabla h$, where $h$ is the height function of $\Sigma$.
\end{itemize}
\end{enumerate}
Then, $\Sigma $ is a piece of one of the complete examples described in \cite{AEG1}.
\end{corolario}

\begin{remark}
The hypothesis $1$ in Theorem \ref{th:KIproducto} is sharp. Indeed, consider the
revolution surface $\Sigma _0$ with positive constant Gaussian curvature in $\r ^3$
given by
$$ \psi (s ,t) = (\sin (s) k(t) , \cos (s) k(t), h(t)) $$with
\begin{equation}\label{ejemplo}
\begin{split}
k(t) &=  b \sin (\sqrt{K}t)\\
h(t) &=   \int _0 ^t \sqrt{1-b^2 K\cos ^2 (\sqrt{K}r)} \, dr
\end{split}
\end{equation}

The complete example is when $b=1$. Also, up to scaling, we can assume that $K=1$.
Thus if we consider $b \in (0,1)$ in \eqref{ejemplo} we obtain a constant Gaussian
curvature (CGC) surface with non vanishing $Q(A,II)$.

We take $\Sigma$ a simply-connected region bounded by two meridians and two
horizontal circles (and therefore lines of curvature), then $\Sigma$ is a compact
embedded disk type CGC bounded by lines of curvature meeting at $4$ vertices with
angle $\pi / 2$.

Similar examples can be constructed in $\s ^2 \times \r$ and $\h ^2 \times \r$ by
constructing a non-complete rotational following the computations in \cite[Section
3]{AEG1}. These surfaces have non vanishing holomorphic quadratic differential
$Q(A,II)$, so we can choose a local conformal parameter (for the second fundamental
form) $z$ so that $Q(A,II)=dz^2$. Thus the piece of the surface corresponding to the
square $\{|{\rm Re} \, z|\leq t_0,\; |{\rm Im} \, z|\leq t_0\}$ in the parameter
domain gives an example of a disk type CGC surface bounded by lines of curvature
(for the associated Codazzi pair) meeting at $4$ vertices with angle equal to $\pi /
2$.
\end{remark}


\section{Proof of the Main Theorem}\label{sec:proof}

We now prove the Theorem \ref{th:choe} stated in Section \ref{sec:choe}. From now
on, $\Sigma $ satisfies the regularity conditions as the previous Section.\\

{\sc{Proof of Theorem \ref{th:choe}}:}\\

Consider on $\tilde\Sigma\supset\Sigma$ the Riemannian metric given by $A$, and let
$z$ be a conformal parameter. Set $\tilde Q$ the $(2,0)-$part of the
complexification of $B$ for the conformal structure given by $A$ on $\tilde \Sigma$, and $Q=\tilde Q _{|\Sigma}$.
Assume, reasoning by contradiction, that $\Sigma$ is not totally umbilical, that is, $Q$ does not vanish identically.


At every non umbilical point in $\tilde{\Sigma}$ there exist two orthogonal (for the metric $A$) lines of curvature,
whereas at an umbilic point (that is, a zero of $\tilde Q$) the lines of curvature bend sharply. Since $\,\mbox{Im} \tilde{Q}=0\,$ on these curves, if we write $\tilde{Q}=f(z)dz^2$ in a neighborhood of $z_0$, the rotation index at an umbilic point $z_0$ is given by
$$I(z_0)=\frac{-1}{4\pi}\delta\mbox{arg}f,$$
where $\delta\mbox{arg}f$ is the variation of the argument of $f$ as we wind once around the singular point.\\

At an interior umbilic point of $\Sigma$, the rotation index of the lines of curvature of $\Sigma$ clearly agrees with the one of $\tilde\Sigma$.
At a point $z_0\in\partial\Sigma$ the rotation index of the lines of curvature of $\Sigma$ is defined as follows (see also \cite{choe}). Consider $\varphi:D_+\to\Sigma$ an
immersion of $D_+=\{\xi\in\c\;:\;|\xi|<1,\mbox{Im}(z)\geq 0\}$ into $\Sigma$,  mapping the diameter of the half disk into
$\partial\Sigma$. The lines of curvature can be pulled back to a line field in $D_+$. Moreover, since the regular curves of
$\partial\Sigma$ are lines of curvature, they can be extended by reflection to a continuous (with singularities) line field on the whole disc.
Thus, we define the rotation index $I(z_0)$ of $Q$ at $z_0\in\partial\Sigma$ to
be half of the rotation index of the extended lines of curvature.

If all the singularities are isolated, the Poincaré-Hopf index theorem gives that $$\sum_{z\in\Sigma}I(z)=1.$$
 The next two claims show that the umbilic points are isolated and give a bound for the rotation index.\\

\begin{quote}
{\bf Claim 1.} \em The zeroes of $Q$ in
$\Sigma\setminus\partial\Sigma$ are isolated, and the rotation index $I(z_0)$ of $X$ at an interior singular point is $\,\leq
-1/2\,$, in particular, it is always negative.\\
\end{quote}

\noindent
{\em Proof of Claim 1.}
Let us see first that the
singularities are isolated. Taking into account Lemma \ref{l1.2} and our hypothesis
1 we infer that
\begin{equation}\label{eq:desig} |\tilde Q_{\bar z}|^2=\frac{\lambda}{2} \frac{||dH||^2}{q} |\tilde Q|^2
 \leq  \frac{\lambda}{2} h^2 |\tilde Q|^2,\end{equation}
on $\tilde\Sigma$, where $z$ is a conformal parameter for the metric $A$, and $\lambda$ is the conformal factor of $A$ in the
parameter $z$. Then, Lemma \ref{lem:cauchy} gives that the zeroes of $\tilde Q$ in
$\tilde\Sigma$ are isolated (recall that we are assuming that $Q$ does not vanish
identically). In particular, the zeroes of $Q$ are isolated in $\Sigma$.
Moreover, locally around a zero $z_0\in\tilde\Sigma$ of $\tilde Q$ we have that
\begin{equation}\label{eq:indneg}
\tilde Q (z)=(z-z_0)^k g(z) dz^2,
\end{equation}
where $k\in\mathbb{N}$ and $g(z)$ is a non-vanishing continuous function.
Therefore, the rotation index is $-k/2\leq
-1/2$, in particular, it is always negative. \hfill{$\Box$}\\

\begin{quote}{\bf Claim 2. } \em The boundary singular points of $X$ are isolated. Moreover, let $z_0\in\partial\Sigma$ be a singular point, then
\begin{enumerate}[(i)]
\item if $z_0$ is not a vertex, its rotation index is $I(z_0)< 0$,
\item if $z_0$ is a vertex of angle $> \pi$, then $I(z_0)< 0$,
\item if $z_0$ is a vertex of angle $<\pi$, then $I(z_0) \leq 1/4$.\\
\end{enumerate}
\end{quote}

\noindent
{\em Proof of Claim 2. }
Consider  $\varphi:D_+\to\Sigma$ a conformal immersion as explained above, with $\varphi(0)=z_0$.
Since $Q$ satisfies $\,\mbox{Im}Q=0\,$ on $\partial\Sigma$, its pull-back can be reflected through the diameter to a continuous quadratic form on the whole unit disc $D$, that will be denoted by $Q^\ast$. Notice that when $z_0$ is a vertex $\varphi'$ could be zero or infinite.

Let $\theta$ be the angle of $\partial\Sigma$ at $z_0$ ($\theta=\pi$ if $z_0$ is not a vertex).
Then $\varphi'$ grows as $|\xi|^{\frac{\theta}{\pi}-1}$ at the origin.
Around $z_0$, $\tilde{Q}$ is given by \eqref{eq:indneg}, although in this case $k$ could be zero, since when $\theta=\pi/2$, $z_0$ is not necessarily  a zero of $\tilde Q$. In particular, $z_0$ is an isolated singularity. Moreover, there are $2(k+2)$ lines of curvature in $\tilde\Sigma$ emanating from $z_0$, and meeting at an equal-angle system of angle $\pi/(k+2)$. In particular, since the curves in $\partial \Sigma$ are lines of curvature, $\theta$ must be a multiple of $\pi/(k+2)$.

If we write $Q^\ast=f(\xi)d\xi^2$ for $\xi\in D$, then
$$f(\xi)=\big(\varphi(\xi)-\varphi(0)\big)^k (\varphi'(\xi))^2 g(\varphi(\xi)), \quad \xi\in D_+.$$
Then the variation of the argument of $f(\xi)$ as we wind once around the origin is $2\theta(k+2)-4\pi$, and the rotation index is $$I^\ast = 1 -\frac{\theta}{2\pi}(k+2).$$

In particular, if  $\theta \geq \pi$, then $I^\ast \leq -k/2 < 0$, whereas for $\theta <\pi$ we have $I^\ast\leq 1/2$ (as $I^\ast<1$, and $2I^\ast$ must be an integer).
Since $I(z_0)=I^\ast/2$, the claim is proved.
\hfill{$\Box$}
\vspace*{0.4cm}

Taking into account the two previous claims, and since the number of vertices of angle $<\pi$ is less than or equal to $3$, we can conclude that
$$\sum_{p\in \Sigma} I(p) \leq 3/4<1,$$
which contradicts the Poincaré-Hopf theorem and shows that $\Sigma$ is totally umbilical.


\end{document}